\documentclass[10pt]{amsart}
\usepackage{amssymb,latexsym}
\usepackage[mathscr]{eucal}

\theoremstyle{plain}
\newtheorem*{Main}{Theorem} 
\newtheorem*{corollary}{Corollary}
\newtheorem{lemma}{Lemma}

\newcommand{\refL}[1]{Lemma~\ref{L:#1}}
\newcommand{\refS}[1]{Section~\ref{S:#1}}
\numberwithin{equation}{section}

\newcommand{\be}[1]{\begin{equation}\label{E:#1}}
\newcommand{\refE}[1]{\eqref{E:#1}}

\newcommand{\q}{\quad}

\newcommand{\lf}{\left\lfloor}
\newcommand{\rf}{\right\rfloor}
\newcommand{\lc}{\left\lceil}
\newcommand{\rc}{\right\rceil}

\newcommand{\ff}{\frac}

\newcommand{\ga}{\alpha}
\newcommand{\gb}{\beta}
\newcommand{\gd}{\delta}
\newcommand{\gre}{\epsilon}

\newcommand{\gf}{\varphi}
\newcommand{\gt}{\tau}

\newcommand{\grg}{\gamma}

\newcommand{\tr}{\{p,q,r\}} \newcommand{\Tr}{\{\, p,q,r\,\}}
\newcommand{\ts}{\{p,q,s\}} \newcommand{\Ts}{\{\, p,q,s\,\}}
\newcommand{\qt}{Q_\gt}
\newcommand{\ca}{\mathcal{A}}
\newcommand{\at}{\ca_\gt}  \newcommand{\At}{A(\gt)}
\newcommand{\Ar}{A(p,q,r)}  \newcommand{\As}{A(p,q,s)}
\newcommand{\rep}{\mathcal{R}_{p,q}}
\newcommand{\gs}{\sigma}
\newcommand{\Sp}{\sigma_p} \newcommand{\Ss}{\sigma_s}
\newcommand{\sn}{\theta(i)}

\begin{document}

\title[Inclusion-Exclusion Polynomials]
{On a Class of  Ternary \\ Inclusion-Exclusion Polynomials}
\author{Gennady Bachman}
\address{University of Nevada, Las Vegas\\
Department of Mathematical Sciences\\
4505 Maryland Parkway\\
Las Vegas, Nevada 89154-4020, USA}
\email{bachman@unlv.nevada.edu}
\author{Pieter Moree}
\address{Max-Planck-Institut f\" ur Mathematic\\
Vivatsgasse 7, D-53111 Bonn, Germany.}
\email{moree@mpim-bonn.mpg.de}
\thanks{We wish to thank Yves Gallot for making available to us
his calculations of heights of inclusion-exclusion polynomials}
\date{March 25, 2010} 

\begin{abstract}
A ternary inclusion-exclusion polynomial is a polynomial
of the form
\[
Q_{\tr}=\ff{(z^{pqr}-1)(z^p-1)(z^q-1)(z^r-1)}
{(z^{pq}-1)(z^{qr}-1)(z^{rp}-1)(z-1)},
\]
where $p$, $q$, and $r$ are integers $\ge3$ and relatively prime
in pairs. This class of polynomials contains, as its principle
subclass, the ternary cyclotomic polynomials corresponding to
restricting $p$, $q$, and $r$ to be distinct odd prime numbers.
Our object here is to continue the investigation of the
relationship between the coefficients of $Q_{\tr}$ and $Q_{\ts}$,
with $r\equiv s\pmod{pq}$. More specifically, we consider the
case where $1\le s<\max(p,q)<r$, and obtain a recursive estimate
for the function $A(p,q,r)$ -- the function that gives the maximum
of the absolute values of the coefficients of $Q_{\tr}$. A simple
corollary of our main result is the following absolute estimate.
If $s\ge1$ and $r\equiv\pm s\pmod{pq}$, then $A(p,q,r)\le s$.
\end{abstract}
\maketitle

\section{Introduction}\label{S:1}

Throughout this paper we adopt the convention that the integers
$p$, $q$, and $r$ are relatively prime in pairs and that
$p,q,r\ge3$. To each set $\gt=\Tr$ we associate a polynomial
$\qt$ given by
\be{1.1}
\qt(z)=\ff{(z^{pqr}-1)(z^p-1)(z^q-1)(z^r-1)}
{(z^{pq}-1)(z^{qr}-1)(z^{rp}-1)(z-1)}.
\end{equation}
A routine application of the inclusion-exclusion principle to
the roots of the factors on the right of \refE{1.1} shows that
$\qt$ is indeed a polynomial and we refer to it as a ternary
(or of order three)
inclusion-exclusion polynomial. This class of polynomials
generalizes the class of ternary cyclotomic polynomials which
corresponds to restricting the parameters $p$, $q$, and $r$ to
be distinct odd prime numbers. As the terminology suggests,
the notion of inclusion-exclusion polynomials is not restricted
to the ternary case, and the reader is
referred to \cite{B1} for an introductory discussion of
inclusion-exclusion polynomials and their relation to cyclotomic
polynomials. Our interest in inclusion-exclusion polynomials is
motivated by the study of coefficients of cyclotomic polynomials.
Thus in the ternary case, the only case we shall consider here,
from a certain perspective, questions about coefficients of
cyclotomic polynomials are really questions about coefficients
of inclusion-exclusion polynomials. We shall see below that
adopting this point of view is rather helpful.

The degree of $\qt$ is
\be{1.2}
\gf(\gt)=(p-1)(q-1)(r-1),
\end{equation}
see \cite{B1}, and we write
\[
\qt(z)=\sum_{m=0}^{\gf(\gt)}a_mz^m\q\bigl[a_m=a_m(\gt)\bigr].
\]
It is plain from \refE{1.1} that $a_m$ are integral. Polynomial
$\qt$ is said to be flat if $a_m$ takes on the values $\pm1$
and 0. The existence of flat $\qt$ with an arbitrary large
$\min(p,q,r)$ was first established in \cite{B2}. This was done
by showing that if
\be{1.3}
q\equiv-1\pmod{p}\q\text{and}\q r\equiv1\pmod{pq}
\end{equation}
then $\qt$ is flat. Actually in \cite{B2} this was stated
explicitly for cyclotomic polynomials only, but the argument used
applies equally well to inclusion-exclusion polynomials. In fact,
this observation extends to much of the work on the coefficients
of ternary cyclotomic polynomials (cyclotomic polynomials of low
order in general - see \cite{B1}) and, in particular, to all such
work referenced in this paper. Consequently, we shall ignore this
distinction in the future and, when appropriate, simply state the
corresponding result for inclusion-exclusion polynomials. An
improvement on \refE{1.3} was obtained by T. Flanagan \cite{Fl}
who replaced both the $-1$ and 1 there by $\pm1$. But the
conditions on $q$ in these results were entirely superfluous, for
it was shown by N. Kaplan \cite{Ka} that
\be{1.4}
\text{$\qt$ is flat if $r\equiv\pm1\pmod{pq}$}.
\end{equation}
Our object here is to establish a general principle of which
\refE{1.4} is seen to be a special case. We begin by introducing
some conventions. Put
\[
\At=\max_m|a_m(\gt)|\q\text{and}\q \at=\{\, a_m(\gt)\,\}.
\]
Moreover, in a slight abuse of notation, let us agree to write
$\Ar$ in place of $\At$ when the dependence of $A$ on
the parameters $p$, $q$, and $r$ needs to be made explicit. Let
us emphasize that, in a departure from the usual practice, we
are not assuming any particular order for the parameters $p$,
$q$, and $r$. The structural symmetry of $\qt$ with respect to
these parameters is a key aspect of the problem and plays an
important role in our development. Correspondingly, we shall
explicitly state any additional assumptions on $p$, $q$, and $r$
when it is appropriate. In his work on \refE{1.4}, Kaplan showed
that for $r>\max(p,q)$, $\At$ is determined completely by the
residue class of $r$ modulo $pq$. More precisely, he showed that
if $r\equiv\pm s\pmod{pq}$ and $r,s>\max(p,q)$ then
\be{1.5}
\Ar=\As.
\end{equation}
Moreover, under the stronger assumption $r,s>pq$, he showed that,
in fact, we have
\be{1.6}
\ca_{\tr}=\begin{cases}
\ca_{\ts},  &\text{if $r\equiv s\pmod{pq}$,}\\
-\ca_{\ts}, &\text{if $r\equiv -s\pmod{pq}$.}
\end{cases}\end{equation}
The first of these identities was also proved by Flanagan
\cite{Fl}. These results gave a strong indication that for
$r>\max(p,q)$, the set $\at$ is also determined completely by
the residue class of $r$ modulo $pq$.

We are thus lead to examine the relation between coefficients of
$Q_{\tr}$ and $Q_{\ts}$ with $r\equiv s\pmod{pq}$. This problem
splits naturally into  two parts according to whether
\be{1.7}
r, s>\max(p,q)\q\text{or}\q r>\max(p,q)>s\ge1.
\end{equation}
say. The first of these cases was dealt with completely by the
first author in \cite{B1}. It was shown there that the identity
\refE{1.6} indeed holds in the full range $r, s>\max(p,q)$. Let
us mention in passing another interesting property of sets $\at$
(see \cite{B1, Bz, GM}):
$\at$ is simply a string of consecutive integers, that is
\[
\at=[\, A^-(\gt),A^+(\gt)\,]\cap\mathbb Z,
\]
where $A^-(\gt)$ and $A^+(\gt)$ denote the smallest and the
largest coefficients of $\qt$, respectively.

That leaves the second alternative in \refE{1.7}, and this case
is the object of the present paper. The statement of our main
result will make use of the following extension of the definition
of $\At$. For $s=1, 2$ and relatively prime in pairs triples
$\Ts$ put
\be{1.8}
\As=s-1.
\end{equation}
We note that this convention is not inappropriate when considered
in the context of the corresponding polynomials $Q_{\ts}$. Indeed,
$Q_{\{p,q,2\}}$ is of order 2 and its coefficients are $\pm1$
and 0, see \cite{B1}, and, as is immediate from \refE{1.1},
$Q_{\{p,q,1\}}(z)\equiv1$.

\begin{Main}
If $r\equiv\pm s\pmod{pq}$ and $r>\max(p,q)>s\ge1$, then
\be{1.9}
\As\le\Ar\le\As+1.
\end{equation}
\end{Main}

Evidently this case is more complicated than the case covered by
\refE{1.5} and, according to the calculations kindly supplied by
Yves Gallot, both possibilities implicit in \refE{1.9} do occur
quite readily. On the other hand, numerical evidence suggests
that in this case too the equality \refE{1.5} is the more likely
outcome. We do not know of any simple criteria that can be used
to determine which of the two possibilities in \refE{1.9}
must hold.

Note that under the hypothesis of the theorem we have, by
\refE{1.5},
\be{1.10}
\Ar=A(p,q,pq\pm s).
\end{equation}
In this light \refE{1.9} is seen as a recursive estimate.
Of course, using an absolute upper bound for $\As$ on the right
of \refE{1.9} yields the corresponding upper bound for $\Ar$.
The corollary below gives a particularly simple estimate of this
type. To get it we use the bound
\be{1.11}
\As\le s-\lc s/4\rc \q[s\ge1],
\end{equation}
proved in \cite{B3} (a better estimate for $\min(p,q,s)\ge7$ was
recently announced by J. Zhao and X. Zhang \cite{ZZ}).

\begin{corollary}
Under the hypothesis of the theorem we have
\be{1.12}
\Ar\le s,
\end{equation}
for all $s\ge1$. Moreover, \refE{1.12} holds with strict
inequality for $s\ge5$.
\end{corollary}

It should be noted that \refE{1.11}, and hence the corollary, hold
with $s$ replaced by $\min(p,q,s)$. Estimate \refE{1.12} sacrifices
precision for convenience and is certainly weaker than the upper
bound of the theorem for $s\ge5$. It is interesting, however, to
consider the quality of this estimate for $s\le4$--the following
observations are based largely on calculations of Yves Gallot.
First we observe that the bound \refE{1.11} is sharp in this range.
For $s=1,2$ this follows by convention \refE{1.8}, and for $s=3,4$
this is verified computationally, e.g., $A(5,7,3)=2$ and
$A(11,13,4)=3$. It follows that for $s\le4$, \refE{1.12} is just
the uniform version of the upper bound of the theorem, and that
the possibility of equality in \refE{1.12} is the only remaining
question. That is, by \refE{1.10}, we are lead to consider the
equation
\be{1.13}
A(p,q,pq+s)=s.
\end{equation}
For $s=1$, \refE{1.13} holds for all choices of $p$ and $q$ since,
trivially, $A(p,q,pq+1)\ge1$. Recall that this is just Kaplan's
result \refE{1.4}. Equation \refE{1.13} also has solutions for
$s=2,3$, for instance $A(3,5,17)=2$ and $A(7,16,7\cdot16+3)=3$. On
the other hand, no solutions were found for $s=4$ with $p, q<100$.

Using the estimate \refE{1.11} carried no penalty for $s\le4$.
For general $s$ we ought to proceed implicitly and use the
function
\[
M(s)=\max_{p,q}\As.
\]
This is well defined by \refE{1.11}. Indeed, using $M(s)$ on the
right of \refE{1.9} gives a sharp form of \refE{1.12} and leads
us to consider the general form of \refE{1.13}, namely the equation
\be{1.14}
A(p,q,pq+s)=M(s)+1.
\end{equation}
The main point is that solutions of \refE{1.14} are particularly
interesting instances of when the upper bound of the theorem is
the best possible. Plainly, the focus here is on the parameter
$s$, and we shall say that $s$ solves \refE{1.14} if the equation
holds for some $\Ts$. Thus we summarize the preceding paragraph
by saying that (i) $M(s)=s-1$, for $s\le4$; and (ii) $s=1, 2, 3$
are solutions of \refE{1.14}. Unfortunately equation \refE{1.14}
takes us into a largely unchartered territory. Indeed, in addition
to the earlier discussion of $s\le4$ we can say with certainty
only that $s=5$ is also a solution. This follows on combining
\refE{1.11} with the explicitly computed
$A(7,11,5)=3=M(5)$ and $A(13,43,13\cdot43+5)=4$.

Finally, observe that for certain types of triples $\gt$, the
application of the theorem may be iterated providing a very
efficient technique for estimating $\At$. For instance, if
$p$ and $q$ are relatively prime we get
\[
A(q,pq\pm1,q(pq\pm1)\pm p)\le A(q,pq\pm1,p)+1=2.
\]

The remainder of this paper gives a proof of the theorem and is
organized as follows. Our proof naturally splits into two parts
corresponding to $s\le2$ (the ``non-ternary case'') and $s\ge3$.
In the next section we collect preliminaries needed for both
cases. The non-ternary case is appreciably simpler and its proof
is carried out in \refS{3}. We include the argument for $s=1$
since it is substantially different from that of \cite{Ka} and
it helps to illuminate the more difficult general argument.
Finally, we complete the proof in \refS{4}.

\section{Preliminaries}\label{S:2}

We begin by observing that given a triple $\Tr$, each integer
$n$ has a unique representation in the form
\be{2.1}\begin{gathered}
n=x_nqr+y_nrp+z_npq+\gd_npqr, \\
0\le x_n<p,\ 0\le y_n<q,\ 0\le z_n<r,\ \gd_n\in\mathbb Z.
\end{gathered}\end{equation}
We shall say that $n$ is ($\gt$-)representable if $\gd_n\ge0$
and let $\chi_{\gt}$ be the characteristic function of
representable integers. When $\gt$ is understood to be fixed we
shall simply write $\chi$ in place of $\chi_{\gt}$. For our
purposes it will be sufficient to consider only $n<pqr$, as we
shall assume henceforth, and in this range the condition
$\gd_n\ge0$ becomes $\gd_n=0$, so that
\be{2.2}
\chi(n)=\begin{cases}
1, &\text{if $\gd_n=0$,}\\
0, &\text{otherwise.}
\end{cases}\end{equation}
The key role of representable integers is evident from the
following identity -- for the proof see \cite{B1}.

\begin{lemma}\label{L:1}
For all $m<pqr$, we have
\be{2.3}
a_m=\sum_{m-p<n\le m}\bigl(\chi(n)-\chi(n-q)-\chi(n-r)
+\chi(n-q-r)\bigr).
\end{equation}
\end{lemma}

Of course, we interpret $a_m$ as 0 for $m<0$ and $m>\gf(\gt)$.
Having the identity \refE{2.3} in the ``extended range'' $m<0$
and, by \refE{1.2}, $\gf(\gt)<m<pqr$ will prove to be useful for
technical reasons.

Recall that we are after a reduction for $\Ar$ with
$r\equiv\pm s\pmod{pq}$, and eventually we shall assume that $r$
satisfies this condition and that $p<q$. Let us emphasize,
however, that unless any of these conditions are used there is
complete symmetry in the parameters $p$, $q$, and $r$. For
instance, \refL{1} implies that \refE{2.3} with $p$ and $r$
interchanged is also valid. When it is not inconvenient, e.g.,
\refE{2.1} and \refE{2.2}, we make this symmetry perfectly
explicit, but we shall opt for convenience, e.g., Lemmas 1 and 2,
whenever this choice has to be made.

\begin{lemma}\label{L:2}
$
\bigl|\chi(n)-\chi(n-p)-\chi(n-q)+\chi(n-p-q)\bigr|\le1.
$
\end{lemma}

\begin{proof}
See \cite[Lemma2]{B3}.
\end{proof}

Note that, by \refE{2.1} and \refE{2.2},
\be{2.4}
\chi(n)=\chi(n-pq)\q\text{unless}\q z_n=\gd_n=0.
\end{equation}
But $z_n=0$ if and only if $n$ is a multiple of $r$, say $n=kr$.
Thus
\be{2.5}
\chi(kr+tpq)=\chi(kr)\q[0\le t<r].
\end{equation}
Similarly
\be{2.6}
\chi(kr-tpq)=0\q[t>0].
\end{equation}
These simple observations are quite handy. Thus our next lemma
\cite[Lemma 3]{B1} is an immediate
consequence of \refL{1} and \refE{2.4}.

\begin{lemma}\label{L:3}
Let
\be{2.7}
I_1=(m-q-p,m-q]\cap\mathbb Z\q\text{and}\q
I_2=(m-p,m]\cap\mathbb Z.
\end{equation}
Then we have
\[
a_m=a_{m-pq},
\]
unless there is $n\in I_1\cup I_2$ such that $n$ is a multiple
of $r$ and either $n$ or $n-r$ are representable.
\end{lemma}

Next we consider \refE{2.1} modulo $pq$ (modulo a product of two
of the parameters). Let $r^*$ be the multiplicative inverse of
$r$ modulo $pq$ and set
\be{2.8}
f(n)=f_{\gt,r}(n)=x_nq+y_np,
\end{equation}
with $x_n$ and $y_n$ given by \refE{2.1}. Then, in the first
place, we have
\be{2.9}
f(n)\equiv nr^*\pmod{pq}.
\end{equation}
Now let $[N]_{pq}$ denote the least nonnegative residue of $N$
modulo $pq$ and let $\rep$ be the set of integers representable
as a nonnegative linear combination of $p$ and $q$, that is,
\be{2.10}
\rep=\{\, N\mid N=xq+yp,\ x,y\ge0\,\}.
\end{equation}
It follows by \refE{2.8}--\refE{2.10} that, in fact,
\be{2.11}
f(n)=\begin{cases} [nr^*]_{pq}, &\text{if $[nr^*]_{pq}\in\rep$}, \\
 [nr^*]_{pq}+pq, &\text{otherwise.}
\end{cases}\end{equation}
There is an obvious advantage in considering linear combinations
in \refE{2.8} over those in \refE{2.1}. This is a useful
observation in view of the following relationship between the
functions $\chi$ and $f$.

\begin{lemma}\label{L:4}
$\chi(n)=1$ if and only if $f(n)\le\lf n/r\rf$.
\end{lemma}
\begin{proof}
See \cite[(3.26)]{B1}
\end{proof}

Our next lemma will be the only observation in this section that
considers triples $\gt=\Tr$ and $\gt'=\Ts$ (with
$r\equiv s\pmod{pq}$) simultaneously. To simplify the notation
we consider $p$ and $q$ to be fixed and write $f_r(n)$ for
$f_{\gt,r}(n)$. In this setting function $f_r(n)$ is a function
of two variables but it depends only on the residue classes
of $r$ and $n$ modulo $pq$.

\begin{lemma}\label{L:5}
If $r\equiv s,\, n\equiv n' \pmod{pq}$, then $f_r(n)=f_s(n')$.
\end{lemma}

\begin{proof}
This is immediate from \refE{2.11}.
\end{proof}

Now put
\be{2.12}
\gs_k(m)=\sum_{m-k<n\le m}\chi(n).
\end{equation}
Then, by \refL{1}, we have
\be{2.13}
a_m=\Sp(m)-\Sp(m-r)-\Sp(m-q)+\Sp(m-q-r).
\end{equation}
For the purpose at hand we shall find it useful to rewrite
this identity as follows.

\begin{lemma}\label{L:6}
If $r=pq+s$ and $s\ge1$, then $a_m=\Sigma_1+\Sigma_2$, with
$\Sigma_i$ given by
\be{2.14}
\Sigma_1=\Ss(m)-\Ss(m-p)-\Ss(m-q)+\Ss(m-q-p)
\end{equation}
and
\be{2.15}
\Sigma_2=\Sp(m-s)-\Sp(m-s-pq)-\Sp(m-q-s)+\Sp(m-q-s-pq).
\end{equation}
\end{lemma}

\begin{proof}
Observe that
\[
\Sp(m)=\Ss(m)+\Sp(m-s)-\Ss(m-p).
\]
Whence
\be{2.16}\begin{aligned}
\Sp(m)-\Sp(m-r) &=\Sp(m)-\Sp(m-s-pq) \\
&=\Ss(m)-\Ss(m-p)+\Sp(m-s)-\Sp(m-s-pq).
\end{aligned}\end{equation}
Of course, \refE{2.16} also holds with $m-q$ in place of $m$.
Combining this with \refE{2.13} proves the claim.
\end{proof}

In the final lemma of this section we evaluate $\Sigma_2$ in
\refE{2.15}. This evaluation depends on whether the two intervals
\be{2.17}
I'_1=(m-s-q-p,m-s-q]\cap\mathbb Z\q\text{and}\q
I'_2=(m-s-p,m-s]\cap\mathbb Z
\end{equation}
contain a multiple of $r$. Note that since $r=pq+s$, the range
$I'_1\cup I'_2$ contains at most one multiple of $r$, which we
will denote by $\ga r$.

\begin{lemma}\label{L:7}
If $r=pq+s$, $s\ge1$, and $p<q$, then
\[
a_m=\begin{cases}
\Sigma_1, &\text{if $\ga r\notin I'_1\cup I'_2$,}\\
\Sigma_1+(-\chi(\ga r))^j, &\text{if $\ga r\in I'_j$,}
\end{cases}\]
with $\Sigma_1$ fiven by \refE{2.14}.
\end{lemma}

\begin{proof}
By \refL{6}, this is just an evaluation of $\Sigma_2$. But
\[
\Sigma_2=\sum_{m-s-p<n\le m-s}\Bigl(\bigl(\chi(n)-\chi(n-pq)\bigr)
-\bigl(\chi(n-q)-\chi(n-q-pq)\bigr)\Bigr).
\]
Therefore, by \refE{2.4}, \refE{2.6}, and \refE{2.17}, we have
\[
\Sigma_2=\begin{cases}
0, &\text{if $\ga r\notin I'_1\cup I'_2$,}\\
(-\chi(\ga r))^j, &\text{if $\ga r\in I'_j$,}
\end{cases}\]
as claimed.
\end{proof}

\section{Proof of theorem: the non-ternary case}\label{S:3}

At this stage we are ready to break the symmetry and, using the
usual convention, put $p<q<r$. Moreover, we note that by
\refE{1.5} it suffices to prove \refE{1.9} for $r=pq+s$, as we
shall assume henceforth.

In this section we deal with $s\le2$. In this case \refE{1.9}
becomes,
\[
s-1\le A(p,q,pq+s)\le s,
\]
by \refE{1.8}, and the first of these inequalities is trivially
satisfied. Therefore to complete the proof in the present case
we need to show that every coefficient $a_m$ of $\qt$ satisfies
\be{3.1}
|a_m|\le s.
\end{equation}
Since inclusion-exclusion polynomials are reciprocal, see
\cite{B1}, we have $a_m=a_{\gf(\gt)-m}$, and it suffices to
prove \refE{3.1} for $m\le\gf(\gt)/2$, as we shall now assume.
Note that by \refE{1.2}, for $m$ in this range the quantity
$\ga r$ occurring in \refL{7} satisfies the condition
\be{3.2}
\ga=\lfloor m/r\rfloor<{\textstyle\ff12}(p-1)(q-1),
\end{equation}
which proves to be quite convenient.

An appeal to \refL{7} leads us to consider two cases. The simplest
case occurs when either $\ga r\notin I'_1\cup I'_2$ or
$\chi(\ga r)=0$. In this case \refL{7} gives
\[
a_m=\Sigma_1=\sum_{m-s<n\le m}
\bigl(\chi(n)-\chi(n-p)-\chi(n-q)+\chi(n-p-q)\bigr),
\]
and the argument is completed by an application of \refL{2}.

Now suppose that $\ga r\in I'_j$ and $\chi(\ga r)=1$. In this case
it is simplest to treat $s=1$ and $s=2$ separately,
and we consider $s=1$ first. Then by \refL{7} we have
\be{3.3}
a_m=\chi(m)-\chi(m-p)-\chi(m-q)+\chi(m-p-q)+(-1)^j.
\end{equation}
We will evaluate this sum using \refL{4}. To this end we observe
that since $\ga r\in I'_j$, every argument of the function $\chi$
occurring in \refE{3.3} is of the form $\ga r+i$ with
$|i|\le p+q$. Moreover, $\ga$ satisfies \refE{3.2}.
But by \refE{2.9}
\[
f(\ga r+i)\equiv\ga+i\pmod{pq},
\]
since both $r$ and its inverse $r^*$ are congruent to 1 modulo
$pq$. It follows from \refL{4} and \refE{2.11} that if
$\chi(\ga r+i)=1$, then, in fact, $f(\ga r+i)=\ga+i$ and $i\le0$.
In particular, $\chi(m)=0$. Furthermore, if $\ga r\in I'_1$,
then we also have $\chi(m-q)=\chi(m-p)=0$, and \refE{3.1} follows
from \refE{3.3}. Moreover, we reach the same conclusion if
$\ga r\in I'_2$ and $\chi(m-p-q)=0$. Finally, \refE{3.3} also
yields \refE{3.1} under the assumptions $\ga r\in I'_2$ and
$\chi(m-p-q)=1$, since in this case we must have $\chi(m-q)=1$.
To see this, write $m-p-q=\ga r-i_0$, so that
\[
q\le i_0<p+q\q\text{and}\q f(m-p-q)=\ga-i_0,
\]
and observe that
\[
[(m-q)r^*]_{pq}=[\ga r-i_0+p]_{pq}=\ga-i_0+p
=x_{m-p-q}q+(y_{m-p-q}+1)p,
\]
by \refE{2.8}. Whence, by \refE{2.11}, $f(m-q)=\ga-i_0+p$
and the desired conclusion follows by \refL{4}. This completes
the proof for $s=1$.

The subcase $s=2$ differs from the previous subcase only in some
technical details. In place of \refE{3.3} we now have, by \refL{7},
\begin{align}
a_m&=\gs_2(m)-\gs_2(m-p)-\gs_2(m-q)
  +\gs_2(m-p-q)+(-1)^j \label{E:3.4}  \\
&=\sum_{m-2<n\le m}\Bigl(\chi(n)-\chi(n-p)-\chi(n-q)
  +\chi(n-p-q)\Bigr) +(-1)^j.   \label{E:3.5}
\end{align}
Since $r=pq+2$, $r^*=(pq+1)/2$ and it is now better to view
arguments of $\chi$ in the form $\ga r+2i+\gre$, with $\gre=0$
or 1. Indeed, by \refE{2.9} we get
\be{3.6}
f(\ga r+2i+\gre)\equiv\ga+i+\gre\ff{pq+1}2\pmod{pq}.
\end{equation}
Now, by \refL{3} we may assume that $\ga r$ is either in
$I_1\cap I'_1$ or in $I_2\cap I'_2$, so that every $\ga r+2i+\gre$
appearing as an argument in \refE{3.5} satisfies $|2i+\gre|<p+q$.
But then, by \refE{3.6}, \refE{3.2}, \refL{4}, and \refE{2.11},
we see that if $\chi(\ga r+2i+\gre)=1$ then we must have $\gre=0$
and $i\le0$. It follows that $\gs_2(m)=0$ and that
\[
\gs_2(m-p), \gs_2(m-q), \gs_2(m-p-q) \le1.
\]
This is sufficient if $\ga r\in I'_2$, for then \refE{3.1} follows
from \refE{3.4}. If on the other hand $\ga r\in I'_1$, then we
also have $\gs_2(m-q)=\gs_2(m-p)=0$, and \refE{3.1} follows in
this case as well. This completes the proof in the non-ternary case.

\section{Proof of theorem: the ternary case}\label{S:4}

Recall from \refS{3} that we fixed $p<q$ and $r=pq+s$. In this
section we will estimate $\At$ in terms of $A(\gt')$, where
$\gt=\Tr$, $\gt'=\Ts$, and $3\le s<q$, and this will require us
to consider coefficients of $\qt$ and $Q_{\gt'}$ simultaneously.
To this end let us adopt the following conventions. We shall
continue to write $a_m$ for coefficients of $\qt$ and we shall
write $b_l$ for coefficients of $Q_{\gt'}$. We shall write $\chi$
and $\chi'$ for the characteristic functions $\chi_{\gt}$ and
$\chi_{\gt'}$ defined in \refE{2.1} and \refE{2.2}, respectively.
We shall also write $\gs_k$ and $\gs'_k$ for the summatory
functions defined in \refE{2.12} with $\chi$ and $\chi'$,
respectively.

Functions $\chi$ and $\chi'$, and hence $\gs$ and $\gs'$, are
closely related. In the next three lemmas we collect certain
properties of these functions.

\begin{lemma}\label{L:9}
If $|j|<s$ then
\[
\chi(kr+j)=\chi'(ks+j).
\]
\end{lemma}

\begin{proof}
Since $f_r(kr+j)=f_s(ks+j)$, by \refL{5}, and
$\lfloor(kr+j)/r\rfloor=\lfloor(ks+j)/s\rfloor$,
the claim follows by \refL{4}.
\end{proof}

\begin{lemma}\label{L:10}
For $|k|<pq$, $0<|j|<s$, and $|\gb|\le\lfloor pq/s\rfloor$,
we have
\[
\chi(kr+j+\gb pq)=\chi(kr+j).
\]
\end{lemma}

\begin{proof}
There is nothing to prove if $\gb=0$, so assume that
$0<|\gb|\le\lfloor pq/s\rfloor$. Recall that $[N]_{pq}$ denotes
the least nonnegative residue of $N$ modulo $pq$ and that
$r^*=s^*$. Write
\be{4.2}
[k+js^*]_{pq}=[k]_{pq}+t_j.
\end{equation}
Evidently
\be{4.3}
\lfloor pq/s\rfloor\le|t_j|\le pq-\lfloor pq/s\rfloor,
\end{equation}
since $t_js\equiv j\pmod{pq}$. Now suppose that $\chi(kr+j)=1$.
Plainly this is not possible unless $k>0$. Therefore in this
case we may replace $[k]_{pq}$ by $k$ in \refE{4.2} and,
by \refE{2.11} and \refL{4}, we get
\be{4.4}
f(kr+j)=k+t_j\q\text{and}\q t_j<0.
\end{equation}
Also, by \refL{5}, $f(kr+j+\gb pq)=f(kr+j)$. But
\[
\Bigl\lfloor\ff{kr+j+\gb pq}r\Bigr\rfloor\ge k-|\gb|
\ge k-\lfloor pq/s\rfloor,
\]
and the claim in this case follows by \refE{4.3} and \refL{4}.

On the other hand
\[
\Bigl\lfloor\ff{kr+j+\gb pq}r\Bigr\rfloor<k+|\gb|
\le k+\lfloor pq/s\rfloor.
\]
Therefore if $\chi(kr+j+\gb pq)=1$ then, by \refL{4},
$f(kr+j)<k+\lfloor pq/s\rfloor$. Arguing as before one readily
verifies that this implies that \refE{4.4} must hold. This
yields $\chi(kr+j)=1$, and the proof is complete.
\end{proof}

\begin{lemma}\label{L:11}
For $|k|<pq$, $0\le\grg<s$, and $|\gb|\le\lfloor pq/s\rfloor$,
we have
\begin{align*}
\Ss(kr+\grg+\gb pq)-\chi(kr+\gb pq) &=\Ss'(ks+\grg)-\chi'(ks) \\
&=\Ss'(ks+\grg-pq).
\end{align*}
\end{lemma}

\begin{proof}
The first identity follows from Lemmas 8 and 9. The second
identity follows from \refE{2.4} and \refE{2.6} with $\chi'$
in place of $\chi$.
\end{proof}

Our preparation is now complete and we are ready to embark on
the main argument. Let $a_m$ be a coefficient of $\qt$ and set
$\ga=\lfloor m/r\rfloor$. Recall from \refS{3} that we may assume
that $\ga$ satisfies \refE{3.2}. Furthermore, by \refL{3}, we
may also assume that
\be{4.5}
\ga r\in I_1\cup I_2.
\end{equation}
Now write
\be{4.6}
\begin{aligned}
m_1&=m=\ga r+\gb_1s+\grg_1, \\ m_3&=m-q=\ga r+\gb_3s+\grg_3,
\end{aligned} \q \begin{aligned}
m_2&=m-p=\ga r+\gb_2s+\grg_2, \\ m_4&=m-p-q=\ga r+\gb_4s+\grg_4,
\end{aligned}\end{equation}
with $0\le\grg_i<s$. Moreover, set
\be{4.7}
l_i=(\ga+\gb_i)s+\grg_i\q\text{and}\q l=l_1
\end{equation}
and observe that
\be{4.8}
l_2=l-p,\ l_3=l-q,\ \text{and}\ l_4=l-p-q.
\end{equation}
From \refE{4.6}, \refE{4.5}, and \refE{2.7} we see that
$s|\gb_i|<p+q+s$, so that
\be{4.9}
|\gb_i|\le\lfloor3q/s\rfloor\le\lf pq/s\rf,
\end{equation}
and, by \refE{3.2},
\be{4.10}
|\ga+\gb_i|<pq.
\end{equation}
Now, by \refE{4.6}, quantities $m_i$ have a representation
in the form
\[
m_i=(\ga+\gb_i)r+\grg_i-\gb_ipq.
\]
Therefore, by \refE{4.10}, \refE{4.9}, \refL{11}, and
\refE{4.7}, we have
\begin{align}
\Ss(m_i)-\chi\bigl((\ga+\gb_i)r-\gb_ipq\bigr)
 &=\Ss'(l_i)-\chi'\bigl((\ga+\gb_i)s\bigr)  \label{E:4.11} \\
 &=\Ss'(l_i-pq).   \label{E:4.12}
\end{align}

We are now in the position to relate the sum $\Sigma_1$ given
by \refE{2.14} and the coefficient $b_l$ of $Q_{\gt'}$ with $l$
given in \refE{4.7}. Using notation \refE{4.6} we write
\be{4.13}
\Sigma_1=\sum_{i=1}^4\sn\Ss(m_i),
\end{equation}
where $\sn=1$, for $i=1,4$, and $\sn=-1$, for $i=2,3$. On the
other hand, by \refE{4.7} and \refE{4.10}, \refL{1} with $r$
replaced by $s$ applies to the coefficient $b_l$. We implement
\refE{2.3} with $\chi'$ in place of $\chi$ and with $p$ and $r$
replaced by $s$ and $p$, respectively, to get
\be{4.14}
b_l=\sum_{i=1}^4\sn\Ss'(l_i),
\end{equation}
by \refE{4.8}. Similarly
\be{4.15}
b_{l-pq}=\sum_{i=1}^4\sn\Ss'(l_i-pq).
\end{equation}
Therefore, by \refE{4.13}, \refE{4.14}, and \refE{4.11}, we have
\[
\Sigma_1-b_l=\sum_{i=1}^4\sn\Bigl(\chi\bigl((\ga+\gb_i)r
-\gb_ipq\bigr)-\chi'\bigl((\ga+\gb_i)s\bigr)\Bigr).
\]
Furthermore applying  \refE{2.5}, \refL{9}, and  \refE{2.6}
to the right side of this expression gives
\be{4.16}\begin{aligned}
\Sigma_1-b_l&=\sum_{\gb_i>0} \sn\Bigl(\chi\bigl((\ga+\gb_i)r
-\gb_ipq\bigr)-\chi'\bigl((\ga+\gb_i)s\bigr)\Bigr)  \\
&=-\sum_{\gb_i>0} \sn\chi'((\ga+\gb_i)s).
\end{aligned}\end{equation}
Moreover if we use \refE{4.15} and \refE{4.12} in place of
\refE{4.14} and \refE{4.11} the same computation yields
\be{4.17}\begin{aligned}
\Sigma_1-b_{l-pq}&=\sum_{i=1}^4
\sn\chi\bigl((\ga+\gb_i)r-\gb_ipq\bigr)\\
&=\sum_{\gb_i\le0} \sn\chi\bigl((\ga+\gb_i)r-\gb_ipq\bigr).
\end{aligned}\end{equation}

We complete the proof by considering the two alternatives of
\refL{7}. Suppose first that $\ga r\notin I'_1\cup I'_2$, with
$I'_j$ given by \refE{2.17}. One then readily verifies that,
in view of \refE{4.5} and \refE{4.6}, we must have either
\be{4.18}
\gb_1=0>\gb_2\ge\gb_3\ge\gb_4,
\end{equation}
if $\ga r\in I_2$, or
\be{4.19}
\gb_4<\gb_3=0<\gb_2\le\gb_1,
\end{equation}
if $\ga r\in I_1$.
In either case $a_m=\Sigma_1$, by \refL{7}, so that \refE{4.16}
holds with $\Sigma_1$ replaced by $a_m$. But under \refE{4.18}
the right side of \refE{4.16} vanishes and we get
\be{4.20}
a_m=b_l.
\end{equation}
Observe that by \refE{4.7}, \refE{3.2}, and \refE{1.2}, index
$l=\ga s+\grg_1$ is arbitrary in the range $l\le\gf(\gt')/2$.
Thus \refE{4.20} says that every integer occurring as a
coefficient of $Q_{\gt'}$ is also a coefficient of $\qt$ and the
first inequality in \refE{1.9} follows.

Now consider \refE{4.19}. In this case \refE{4.16} yields
\[
a_m-b_l=\chi'((\ga+\gb_2)s)-\chi'((\ga+\gb_1)s),
\]
so that we certainly have
\be{4.21}
|a_m-b_l|\le1.
\end{equation}
But \refE{4.20} and \refE{4.21} imply the right side of \refE{1.9},
and it only remain to consider the second alternative of \refL{7}.

Suppose now that, in addition to \refE{4.5},
 $\ga r\in I'_1\cup I'_2$. Then, reasoning as in
\refE{4.18} and \refE{4.19}, we conclude that either
\be{4.22}
\gb_4\le\gb_3\le\gb_2\le0<\gb_1,
\end{equation}
if $\ga r\in I'_2$, or
\be{4.23}
\gb_4\le0<\gb_3\le\gb_2\le\gb_1,
\end{equation}
if $\ga r\in I'_1$. In the first case we get,
by \refL{7}, \refE{4.16}, and
\refE{4.22},
\[
a_m-b_l=\chi(\ga r)-\chi'((\ga+\gb_1)s).
\]
Therefore \refE{4.21} holds in this case as well. In the
second case we appeal to \refE{4.17}
instead of \refE{4.16} to get, by \refL{7} and \refE{4.23},
\[
a_m-b_{l-pq}=\chi\bigl((\ga+\gb_4)r-\gb_4pq\bigr)-\chi(\ga r).
\]
This yields \refE{4.21} with $b_l$ replaced by $b_{l-pq}$,
and completes the proof of the theorem.

\end{document}